\documentclass[letterpaper, 10 pt, conference]{ieeeconf}  

\usepackage[utf8]{inputenc}
\usepackage{csquotes}
\usepackage{amsmath,amsfonts,amssymb}
\usepackage{graphicx}
\usepackage{array,multirow,makecell}
\usepackage[T1]{fontenc} 
\graphicspath{ {./ } }

\usepackage[pdftex,bookmarks,colorlinks,breaklinks]{hyperref}
\hypersetup{colorlinks=true, linkcolor=red, filecolor=blue, pagecolor=blue, urlcolor=blue, citecolor=blue}

\usepackage[nolabel, final]{showlabels}
\usepackage{color}

\newcommand{\comment}[1]{}

\def\Rset{\mathbb{R}}
\def\Nset{\mathbb{N}}
\def\cR{{\cal R}}

\def\bA{{\bar A}}
\def\bB{{\bar B}}
\def\bE{\bar{E}}
\def\bI{\bar{I}}
\def\bR{\bar{R}}
\def\bS{\bar{S}}

\usepackage{bbding}   
\definecolor{vert}{rgb}{0.06, 0.7, 0.6}   
\definecolor{mauve}{rgb}{0.6, 0.2, 0.99}   
\definecolor{bleu}{rgb}{0, 0.45, 1}   

\newtheorem{Def}{Definition}
\newtheorem{Thm}{Theorem}
\newtheorem{Lem}{Lemma}
\newtheorem{Rem}{Remark}
\newtheorem{Cor}{Corollary}

\IEEEoverridecommandlockouts                              

\title{\LARGE \bf
Modelling, Analysis, Observability and Identifiability of Epidemic Dynamics with Reinfections}

\author{Marcel Fang$^{1}$ and Pierre-Alexandre Bliman \Envelope$^{\, 1}$
\thanks{The authors acknowledge financial support from the MODCOV19 platform.}
\thanks{$^{1}$Marcel Fang and Pierre-Alexandre Bliman are with Sorbonne Universit\'e, Laboratoire Jacques-Louis Lions, CNRS and with Inria, France.
        {\tt\small  \{marcel.fang,pierre-alexandre.bliman\}@inria.fr }}
}

\begin{document}

\maketitle
\thispagestyle{empty}
\pagestyle{empty}

\begin{abstract}
We consider in this paper a general SEIRS model describing the dynamics of an infectious disease including latency, waning immunity and infection-induced mortality.
 We derive an infinite system of differential equations that provides an image of the same infection process, but counting also the reinfections.
 Existence and uniqueness of the corresponding Cauchy problem is established in a suitable space of sequence valued functions, and the asymptotic behavior of the solutions is characterized, according to the value of the basic reproduction number.
 This allows to determine several mean numbers of reinfections related to the population at endemic equilibrium.
We then show how using jointly measurement of the number of infected individuals and of the number of primo-infected provides observability and identifiability to a simple SIS model for which none of these two measures is sufficient to ensure on its own the same properties. 


\end{abstract}


\section{INTRODUCTION}

Since their introduction by Kermack and McKendrik in 1927 \cite{Kermack1927}, compartmental models have been massively used in mathematical epidemiology in order to study epidemic dynamics.
The obtained dynamical models may be analyzed and simulated, with parameter values 
estimated by fitting 
to observed data.
The inverse problem consisting of this estimation process is essential for realistic replication of the phenomenon.
It is thus important to look beforehand if the obtained parameter estimates are meaningful, 
and first of all whether perfect, error-free, measurement of the system actually contains information on the unknown parameters  --- in other terms whether the model is {\em identifiable} \cite{Jacquez:1985aa}.
Identifiability is only a recent topic in mathematical epidemiology, with few works addressing that issue.
A survey on this topic of has been recently published~\cite{Hamelin:2020aa}.
Xia and Moog \cite{Xia:2003ug} were among the first who considered this question, in a paper on an intra-host model of HIV.
Structural identifiability \cite{Jacquez:1985aa,Miao:2011tj} for the classical SIR and SEIR models, based on prevalence measurement, has been studied by Tuncer et al.\ \cite{Tuncer:2018up}.
One may also cite Evan et al.\ \cite{Evans:2005wj} who addressed the identifiability problem for a SIR model with seasonal forcing, with either prevalence or incidence measured; or \cite{Eisenberg:2013vn} for a SIR model for waterborne disease.

At the same time, 
the phenomena of reinfection, and particularly the counting of the number of reinfections, have been little studied to date.
Among the works addressing that question, 
Andreasen et al.\ \cite{Andreasen1997} and Abu-Raddad \& Ferguson \cite{Abu2005}  studied models with reinfections by different strains.
Arino et al.\ \cite{Arino2003} presented a SVIRS model in order to analyze the efficiency of vaccination. With the same goal, Gomes et al.\ \cite{Gomes2004} studied systematically different SIRS models with vaccination, partial and temporary immunity.
In addition, Katriel \cite{Katriel2010} highlighted a threshold condition for endemicity for a SIRS system borrowed from \cite{Gomes2004}.
This author also proposed in the same paper a modified SIRS system with an infinite set of differential equations capable of counting the number of reinfections --- to our knowledge, the only contribution made from this perspective. 

In the present article, we draw inspiration from Katriel's modelling approach, with the general aim of analyzing whether measuring the number of reinfections may provide more information for observability and identifiability
than the usual perspective, which only considers all the infections globally.
For this, we propose and analyze a SEIRS differential system with infinite number of equations that takes into account the reinfections.
The model is presented in Section \ref{se2} and its well-posedness is established in Section \ref{se3}.
The asymptotic convergence of the solutions is then studied in Section \ref{se4}, while quantities of interest related to the asymptotic mean numbers of reinfections at the endemic equilibrium are computed in Section \ref{se5}.
Simulations illustrating the behavior of the system are presented in Section \ref{se6}.
Finally, we present in Section \ref{se7} results that demonstrate how the supplementary information on the number of reinfections may render observable and identifiable a SIS model which otherwise possesses none of these properties.

\section{A SEIRS MODEL COUNTING REINFECTIONS}
\label{se2}

We begin by introducing the classical SEIRS system, presented e.g.~in the recent paper \cite{Bjornstad2020} by Bj\o rnstad et al.~and depicted in Figure \ref{fi1}.
It is written as follows :
\begin{subequations}
\label{eqn:SEIRS_global}
\begin{align}
\dot S &= b N -\beta S\frac{I}{N} + \omega R - \mu S,\\
\dot E &= \beta S\frac{I}{N} - (\sigma+\mu)E,\\
\dot I &= \sigma E - (\gamma+\mu+\nu)I,\\
\dot R &= \gamma I - (\omega+\mu) R,
\end{align}
\end{subequations}
where $N(t)=S(t)+E(t)+I(t)+R(t)$.
Here the variables $S(t), E(t), I(t), R(t)$ represent respectively the number of individuals that are {\em susceptible}, {\em exposed} to the disease after being infected, {\em infectious} and {\em recovered} and subject, at least transiently, to immunity.
The number $N(t)$ is the total population size.
All the model coefficients are nonnegative, with $b$ and $\mu$ representing the birth and natural mortality rates, while the other coefficients are characteristic of the considered disease.
The coefficient $\beta$ is the contact rate, and $\omega^{-1}$, $\sigma^{-1}$, $\gamma^{-1}$ correspond respectively to the period of immunity, 
the period of infection and the period of latency when the subject is infected but not yet contagious.
Last, the constant $\nu$ is the infection-induced mortality rate.

\begin{Rem}
\label{re1}
Note that generally speaking, the total population $N=S+E+I+R$ may vary, as
$\dot N = (b-\mu)N -\nu I$.
In particular, when $\nu=0$, the solution of \eqref{eqn:SEIRS_global} diverges, resp.~vanishes, as $t$ tends to $+\infty$ when $b>\mu$, resp.~$b<\mu$.
When $b=\mu$ and $\nu>0$, the variable $N$ may also vary.
\hfill $\square$
\end{Rem}

\begin{figure}[h!]
\centerline{\includegraphics[scale=0.32]{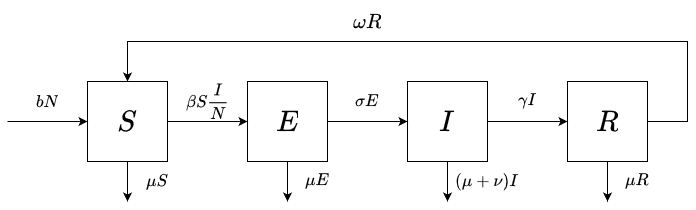}}
\caption{ Flowchart of system \eqref{eqn:SEIRS_global}}
\label{fi1}
\end{figure}

For non-permanent immunity, that is when $\omega>0$, individuals recovered become newly susceptible after healing.
We now want to account for these reinfections, by keeping track of their number.
Denoting, for any $i\in\Nset$, $E_i(t)$ the number of individuals exposed after having been ill $i-1$ times; $I_i(t)$ the number of individuals infected for the $i$-th time; $R_i(t)$ the number of individuals immune after having recovered $i$ times; and $S_i(t)$ the number of susceptible individuals that lost their transitory immunity after $i-1$ recoveries, we expand system \eqref{eqn:SEIRS_global} along the idea used in \cite{Katriel2010}, by dividing the compartments and adding an index associated to the number of reinfections.
The corresponding system is schematized on Figure \ref{fi2}, and the following equations are obtained:
\begin{subequations}
\label{eqn:SEIRS} 
\begin{align}
&\dot S_i = \omega R_{i-1} -\beta  S_{i}\frac{I}{N} - \mu S_i, \quad i\geq 1,\\
&\dot E_i = \beta S_{i}\frac{I}{N} - (\sigma+\mu)E_i, \quad i\geq 1,\\
&\dot I_i = \sigma E_i - (\gamma+\mu+\nu)I_i, \quad i\geq 1,\\
&\dot R_i = \gamma I_i - (\omega+\mu) R_i, \quad i\geq 1,
\end{align}
\end{subequations}
with 
$S(t) := \sum_{i\geq 1} S_i(t)$, $E(t) := \sum_{i\geq 1} E_i(t)$, $I(t) := \sum_{i\geq 1} I_i(t)$, $R(t) := \sum_{i\geq 1} R_i(t)$, $N(t) := S(t)+E(t)+I(t)+R(t)$.
Furthermore by convention one puts
\begin{equation*}
\omega R_0(t)=b N(t),
\end{equation*}
representing the recruitment term.
Finally, the initial condition for the Cauchy problem associated to \eqref{eqn:SEIRS} is given by the quantities $S_i(0)=S_i^0$, $E_i(0)=E_i^0$, $I_i(0)=I_i^0$, $R_i(0)=R_i^0$, for $i\geq 1$.
As depicted by the unwrapping of Figure \ref{fi1} in Figure \ref{fi2}, the central difference is that, after an $i$-th recovery, the individuals enter a new susceptible compartment $S_{i+1}$, instead of coming back to a unique reservoir $S$ as in \eqref{eqn:SEIRS_global}.

Notice that summing up the equations in \eqref{eqn:SEIRS} and the initial conditions, one formally recovers system \eqref{eqn:SEIRS_global}.
This property will be elucidated afterwards.
 In the sequel, we 
 call \eqref{eqn:SEIRS_global} the {\em macroscopic} system and \eqref{eqn:SEIRS} the {\em microscopic} one, as \eqref{eqn:SEIRS} disentangles the hidden reinfection structure of the former.
  
Let us now introduce some notations.
For $n\in \Nset$, define respectively by $X^n, X^n_+, X^n_{++}$ the spaces of sequences :
$\underbrace{l^1 \times \dots \times l^1}_{\text{$n$ times}}$, $\underbrace{l^1_+\times \dots \times l^1_+}_{\text{$n$ times}}$, $\underbrace{l^1_{++}\times \dots \times l^1_{++}}_{\text{$n$ times}}$,
where $l^1$ is the Banach space of summable sequence, $l^1_+\subset l^1$ the subspace of $l^1$ sequences of  nonnegative numbers and $l^1_{++}\subset l^1_+$ the subspace of sequences of positive numbers.
The space $X^n$ is endowed with the norm $\|x\|_{X^n}:=\sum_{1\leq i\leq n}\|x_i\|_{l^1}$ for  $x=(x_1,x_2,...,x_n)\in X^n$.
\begin{figure}[h!]
\centerline{\includegraphics[scale=0.30]{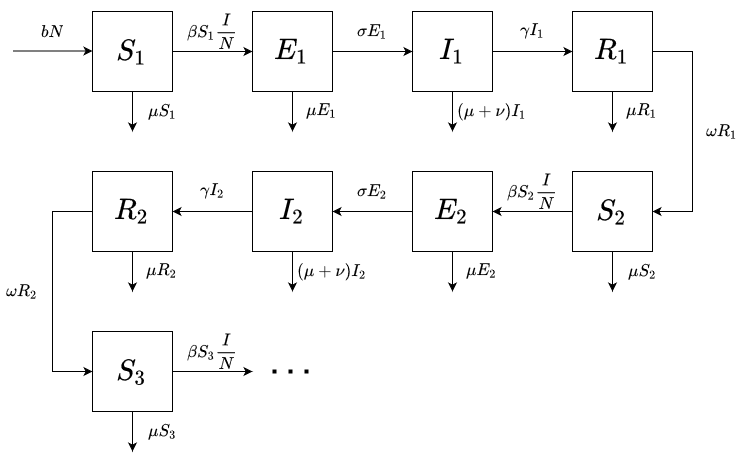}}
\caption{Flowchart of system \eqref{eqn:SEIRS}}
\label{fi2}
\end{figure}
We assume in the whole paper that the initial condition $(S^0_i, E^0_i, I_i^0, R_i^0)_{i\geq1}$ belongs to $X_+^4$ and moreover that the system contains initially some infected (without which the solution is trivial), that is
\begin{align}
\label{eq73}
\|I^0_i\|_{l^1}+\|E^0_i\|_{l^1}=\sum_{i\geq1} (I^0_i + E^0_i) > 0.
\end{align}

\section{WELL-POSEDNESS}
\label{se3}

Due to the infinite dimension of system \eqref{eqn:SEIRS}, proving its well-posedness is not completely evident.
We follow here  an approach employed for example for the study of the Becker-D\"oring system, see e.g.~\cite{Ball1986,Laurencot2002,Doumic2019}.
One first defines an adequate notion of solution for the Cauchy problem associated to \eqref{eqn:SEIRS}. 
\begin{Def}
\label{de1}
Let $0<T\leq \infty$ and $x_0:=(S_i^0,E_i^0, I_i^0,R_i^0)_{i\geq0}\in X^4_+$ verifying \eqref{eq73}.
We call {\em solution of \eqref{eqn:SEIRS} on $[0,T)$} any function $x:[0,T)\rightarrow X^4_+,~ t\mapsto x(t) :=(S_i(t),E_i(t),I_i(t),R_i(t))_{i\geq 1}$ such that :
\begin{enumerate}
    \item each function $S_i, E_i, I_i, R_i : [0,T)\rightarrow \Rset^+ \times \Rset^+$, $i\geq 1$, is continuous and $\sup_{t\in[0,T)} \| x(t) \|_{X^4}<\infty$,
    \item For all $t\in [0,T)$, $i\geq1$,
\vspace{-.3cm}
\begin{gather*} 
\hspace{-.5cm}
S_i(t)=S_i(0)+\int_0^t (\omega R_{i-1}(s)-(\beta \frac{I(s)}{N(s)}+\mu) S_i(s))ds, \nonumber\\
\hspace{-.5cm}
E_i(t)=E_i(0)+\int_0^t (\beta S_{i}(s)\frac{I(s)}{N(s)}-(\sigma+\mu) E_i(s))ds,\nonumber \\
\hspace{-.5cm}
I_i(t)=I_i(0)+\int_0^t(\sigma E_i(s)-(\mu+\gamma+\nu)I_i(s))ds,\nonumber\\
\hspace{-.3cm}
R_i(t)=R_i(0)+\int_0^t (\gamma I_i(s)-(\omega+\mu) R_i(s))ds. \nonumber
\qquad \hfil\square
\end{gather*}
\end{enumerate}
\end{Def}

Notice that the condition $\sup_{t\in[0,T)} \| x(t) \|_{X^4}<\infty$ implies that the functions $I$, and $N$ are bounded and by definition we have $I\leq N$.
Then by Lebesgue's dominated convergence theorem, the quantities $I$, $N$ and $\frac{I}{N}$ are integrable and the integral formulas are well defined.

The following result establishes the well-posedness of system \eqref{eqn:SEIRS}.
\begin{Thm}
\label{th6}
For all $x_0\in X^4_+$ verifying \eqref{eq73} and $ T \in (0,+\infty)$, there exists a unique solution $x=(S_i,E_i, I_i, R_i)_{i\geq 1}$ of \eqref{eqn:SEIRS} on $[0,T)$.
Moreover,
\begin{enumerate}
\item
$x(t)\in X^4_{++}$ for all $t\in (0,T)$;
\item
$(S, E, I, R)$ is continuously differentiable on  $(0,T)$ and fulfils \eqref{eqn:SEIRS_global}, where $S := \sum_{i\geq 1} S_i$, $E := \sum_{i\geq 1} E_i$, $I := \sum_{i\geq 1} I_i$, $R := \sum_{i\geq 1} R_i$;
\item
for all $i\geq 1$, $(S_i, E_i, I_i, R_i)$ is continuously differentiable and fulfils \eqref{eqn:SEIRS} on $(0,T)$.
    \hfill $\square$
\end{enumerate}
\end{Thm}

Notice that, due to its triangular structure, the truncation of system \eqref{eqn:SEIRS} obtained by considering equation \eqref{eqn:SEIRS_global} together with the equations in \eqref{eqn:SEIRS} for $i=1,\dots,n$, for a given $n\in\Nset$, yields an {\em exact} computation of all these variables.
This is used in the simulations presented in Section \ref{se6}.

\begin{proof}
The proof is based on the successive use of the following lemmas.
\begin{Lem}\label{le4}
For all $T\in(0,\infty)$, there exists a solution  $x=(S_i,E_i,I_i,R_i)_{i\geq 1}$ of \eqref{eqn:SEIRS} on $[0,T)$, with values in $X^4_+$.
For this solution, $(S, E, I, R)$ fulfils \eqref{eqn:SEIRS_global}, and $S(t), E(t), I(t), R(t)>0$ for all $t\in (0,T)$.
\hfill $\square$
\end{Lem}

\begin{Lem}
\label{le5}
Every solution of \eqref{eqn:SEIRS} on $[0,T)$ is such that $(S,E,I,R)$ fulfils \eqref{eqn:SEIRS_global} everywhere on $[0,T)$.
\hfill $\square$
\end{Lem}

\begin{Lem}
\label{le6}
Let $x=(S_i, E_i, I_i, R_i)_{i\geq 1}$, $x'=(S'_i, E'_i,I'_i, R'_i)_{i\geq 1}$ be two solutions of \eqref{eqn:SEIRS} on $[0,T)$ with values in $X_+^4$, then $x=x'$.
\hfill $\square$
\end{Lem}

\begin{Lem}
\label{le7}
Let $x=(S_i, E_i, I_i, R_i)_{i\geq 1}$ be a solution of \eqref{eqn:SEIRS} on $[0,T)$, then $x(t)\in X^4_{++}$ for all $t>0$ and fulfils the differential equations \eqref{eqn:SEIRS} everywhere on $(0,T)$.
\hfill $\square$
\end{Lem}

Due to space limitation, we only provide a sketch of the proofs.
Lemma \ref{le4} is proved by constructing explicitly a solution, obtained as solution of the {\em triangular} linear system
\begin{gather*}
    \dot S_i = \omega R_{i-1} -\beta  S_{i}\frac{\imath}{n} - \mu S_i, \quad i\geq 1,\\
    \dot E_i = \beta S_{i}\frac{\imath}{n} - (\sigma+\mu)E_i, \quad i\geq 1,\\
    \dot I_i = \sigma E_i - (\gamma+\mu+\nu)I_i, \quad i\geq 1,\\
    \dot R_i = \gamma I_i - (\omega+\mu) R_i, \quad i\geq 1,
\end{gather*}
with the same initial condition as in \eqref{eqn:SEIRS}, where $\imath$ and $n:=s+e+\imath+r$ come from the unique solution $(s,e,\imath, r)$ of system \eqref{eqn:SEIRS_global} on $[0,T)$, with initial condition
\[
s^0:=\sum_{i\geq1}S^0_i,e^0:=\sum_{i\geq1}E^0_i,\imath^0:=\sum_{i\geq1}I^0_i,r^0:=\sum_{i\geq1}R^0_i.
\]

To show that the function $(s,e,\imath,r)$ constructed in this way is identical to the series $(\sum_iS_i$, $\sum_iE_i$, $\sum_iI_i$, $\sum_iR_i)$ (denoted $(S,E,I,R)$ in the sequel), one first proves, by applying a comparison result, that, for any $t\in (0,T)$,
\begin{gather*}
0 \leq s_n(t):=\sum_{i=1}^nS_i(t)\leq s(t),\ 0 \leq e_n(t):=\sum_{i=1}^nE_i(t)\leq e(t), \\
0 \leq  \imath_n(t):=\sum_{i=1}^nI_i(t)\leq \imath(t),\ 0 \leq r_n(t):=\sum_{i=1}^nR_i(t)\leq r(t).
\end{gather*}
This provides pointwise convergence of the partial sums $\{s_n(t)\}, \{e_n(t)\}, \{\imath_n(t)\}, \{r_n(t)\}$, towards limits denoted $S(t),E(t),I(t),R(t)$.

Uniform boundedness in the previous identities yields equicontinuity of the previous sequences, from which ones deduces, thanks to Arzelà–Ascoli theorem \cite{Brezis2010}, the continuity of the limits $S(t),E(t),I(t),R(t)$.
Using Dini's theorem \cite{Bartle2000}, the convergence of these positive series is indeed uniform, as well as the convergence of their derivatives $\{\dot s_n(t)\}, \{\dot e_n(t)\}, \{\dot \imath_n(t)\}, \{\dot r_n(t)\}$, and this allows to interchange summation and differentiation.
This shows that $(S,E,I,R) = (s, e, \imath, r)$, and that $(S_i,E_i,I_i,R_i)$ is a solution in the sense of Definition \ref{de1}.

To show Lemma \ref{le5}, let $x=(S_i, E_i, I_i, R_i)_{i\geq 1}$ be solution of \eqref{eqn:SEIRS}.
Using the bound
$\sup\{\| x(t) \|_{X^4}\ :\ t\in[0,T)\} <\infty$
yields equicontinuity, and arguments similar to those developed in the proof of Lemma \ref{le4} shows that  the partial sums $s_n, e_n, \imath_n$ and $r_n$ defined as previously converge uniformly to continuous limit functions.
The integrands in the formula of the quantities $S_i,E_i,I_i, R_i$ are then continuous, and this allows to differentiate these quantities, yielding \eqref{eqn:SEIRS}.
Using now the uniform convergence of the series obtained previously and summing up the equations of \eqref{eqn:SEIRS}, we deduce again the uniform convergence of the derivatives $\dot s_n, \dot e_n, \dot \imath_n, \dot r_n$.
Interchanging summation and differentiation is then possible which ensures that $(\dot S, \dot E, \dot I, \dot R)$ fulfils \eqref{eqn:SEIRS_global}.


Notice that, once the solution of \eqref{eqn:SEIRS_global} has been obtained, the functions $S_i, E_i, R_i, I_i$, $i\geq 1$, are deduced in the same unique way as the solution of the triangular linear system in the proof of Lemma \ref{le4}.
This establishes uniqueness, and proves Lemma \ref{le6}.

Due to the regularity of $(S,E,I,R)$ provided by Lemma \ref{le5}, the same point of view shows that for any solution of \eqref{eqn:SEIRS} (in the sense of Definition \ref{de1}), one may differentiate the integral formulas in Definition \ref{de1}, thus demonstrating that $x$ verifies the differential equations \eqref{eqn:SEIRS} everywhere on $(0,T)$, as stated in Lemma \ref{le7}.


Last, positivity of the solution is obtained by using monotonicity (or positivity) of system \eqref{eqn:SEIRS}, see \cite{Smith1995,Farina2000}, achieving the proof of Lemma \ref{le7}.
Using these four lemmas allows finally to establish Theorem \ref{th6}.  
\end{proof}

\section{ASYMPTOTIC BEHAVIOR}
\label{se4}

The existence and uniqueness of the solution being demonstrated, let us focus on its asymptotic behavior, in order to unravel the underlying structure of population (in terms of numbers of reinfections) at the endemic equilibrium.
Considering the fact that the solution of system \eqref{eqn:SEIRS_global}, and also system \eqref{eqn:SEIRS}, may diverge as $t$ tends to $+\infty$ (see Remark~\ref{re1}),
we first normalize \eqref{eqn:SEIRS}, in order to study the relevant limits.
For any $A\in\{S,E,I,R\}$, introduce the normalized quantities
\[
\bA(t) := \frac{A(t)}{N(t)},\qquad
\bA_i(t) := \frac{A_i(t)}{N(t)}.
\]
One gets from \eqref{eqn:SEIRS}:
\vspace{-.2cm}
\begin{subequations}
\label{eqn:SEIRS_norm}
\begin{align}
&  \dot {\bar S}_i = \frac{\dot S_i}{N}-\frac{S_i}{N^2}\dot N =\omega \bar R_{i-1} -(\beta-\nu)  \bar S_i\bar I  - b \bar S_i,\\
& \dot {\bar E}_i = \frac{\dot E_i}{N}-\frac{E_i}{N^2}\dot N =\beta \bar S_i\bar I - (\sigma+b)\bar E_i +\nu \bar I\bar E_i,\\
& \dot {\bar I}_i =\frac{\dot I_i}{N}-\frac{I_i}{N^2}\dot N = \sigma \bar E_i - (\gamma+b+\nu)\bar I_i + \nu \bar I\bar I_i,\\
& \dot {\bar R}_i = \frac{\dot R_i}{N}-\frac{R_i}{N^2}\dot N = \gamma \bar I_i - (\omega+b) \bar R_i + \nu \bar{I}\bar R_i,
\end{align}
\end{subequations}
with $\omega \bar R_0=b$.
Normalizing the macroscopic quantities gives
\vspace{-.4cm}
\begin{subequations}
\label{eqn:SEIRS_normglobal}
\begin{gather}
\dot {\bar{S}} = \frac{\dot S}{N}-\frac{S}{N^2}\dot N =b -(\beta-\nu) \bar S\bar I + \omega  \bar R - b \bar S, \\
\dot {\bar{E}} = \frac{\dot E}{N}-\frac{E}{N^2}\dot N =\beta \bar S\bar I - (\sigma+b)\bar E + \nu \bar I \bar E,\\
\dot {\bar{I}} = \frac{\dot I}{N}-\frac{I}{N^2}\dot N =\sigma \bar E - (\gamma+b+\nu)\bar I +\nu \bar I^2,\\
\dot {\bar{R}} = \frac{\dot R}{N}-\frac{R}{N^2}\dot N =\gamma \bar I - (\omega+b)\bar R + \nu \bar I\bar R.
\end{gather}
\end{subequations}

Thanks to Theorem {\rm\ref{th6}}, existence and uniqueness of solution for system \eqref{eqn:SEIRS} is guaranteed on $[0,+\infty)$.
Moreover we have, for this solution, $N(t)=\|x(t)\|_{X^4}>0$ for all $t\geq0$.
One then easily obtains, by dividing the solution of \eqref{eqn:SEIRS} by $N(t)$, existence and uniqueness result for the normalized system \eqref{eqn:SEIRS_norm}.

\begin{Rem}
Let the set $\Gamma=\{\bar x\in X^4_+ : \|\bar x\|_{X^4}=1\}$.
As $\dot {\bar{S}} + \dot {\bar{E}} + \dot {\bar{I}} + \dot {\bar{R}} = (b-\nu\bar I) (1-\bar S-\bar E-\bar I-\bar R) \equiv 0$, the set $\Gamma$ is positively invariant for the normalized system \eqref{eqn:SEIRS_norm}, as expected.
\hfill $\square$
\end{Rem}

\begin{Rem}
\label{rem6}
Notice that the normalized systems \eqref{eqn:SEIRS_norm} and \eqref{eqn:SEIRS_normglobal} do {\em not} depend upon the value of $\mu$.
\hfill $\square$
\end{Rem}

We have the following result for 
systems \eqref{eqn:SEIRS_norm} and \eqref{eqn:SEIRS_normglobal}.
\begin{Thm}
\label{th7}
The basic reproduction number of system \eqref{eqn:SEIRS_normglobal} is
$\cR_0 := \frac{\sigma}{\sigma+b}\frac{\beta}{\gamma+\nu+b}$.

Moreover, for any $\bar x_0\in \Gamma$ verifying \eqref{eq73},
the solutions 
of \eqref{eqn:SEIRS_norm} and \eqref{eqn:SEIRS_normglobal} are such that :
\begin{enumerate}
\item
if $\cR_0 < 1$, denoting $\delta$ the Kronecker delta, one has for any $i\geq 1$,
\begin{gather*}
\hspace{-.7cm}
\lim_{t\to +\infty} (\bar S_i(t), \bar E_i(t),\bar I_i(t),\bar R_i(t)) = \delta_i^1 (1,0,0,0),\\
\lim_{t\to +\infty} (\bar S(t),\bar E(t),\bar I(t),\bar R(t)) = (1,0,0,0).
\end{gather*}
\item
if $\cR_0 > 1$, there exists a unique nonzero equilibrium $(\bar S^*,\bar E^*,\bar I^*,\bar R^*)$ of \eqref{eqn:SEIRS_normglobal}, which is indeed positive.
The quantity $\bar I^*$ satisfies:
\begin{subequations}
\begin{multline}
\label{eq85}
    \left(\frac{\beta-\nu}{b}\bar I^*+1\right)\left(1-\frac{\nu \bar I^*}{\sigma+b}\right)\left(1-\frac{\nu\bar I^*}{\gamma+\nu+b}\right)\\
    =    \cR_0\left(1+\frac{\gamma}{b}\frac{\omega}{\omega+b-\nu\bar I^*}\right),
\end{multline}
and the three other quantities $\bar S^*, \bar E^*, \bar R^*$ are given as
\begin{multline}
\label{eq85bis}
\hspace{-.7cm}
\bar S^* = \frac{\gamma+b+\nu-\nu \bar I^*}{\sigma} \frac{\sigma+b-\nu \bar I^*}{\beta},\
\bar R^*=\frac{\gamma}{\omega+b-\nu \bar I^*}\bar I^*,\\
\bar E^*=\frac{\gamma+\nu+b-\nu \bar I^*}{\sigma}\bar I^*.
\end{multline}
\end{subequations}

Furthermore, let 
\begin{gather*}
\phi :=\frac{\omega}{(\beta-\nu) \bar I^*+b}\frac{\gamma}{\omega+b-\nu \bar I^*}\frac{\bar I^*}{\bar S^*},
\end{gather*}
then $0<\phi<1$ and the following asymptotic convergence property holds, for every $i\geq 1$:
\begin{gather*} 
\hspace{-.3cm}
\lim_{t\to +\infty} (\bar S_i(t), \bar E_i(t), \bar I_i(t), \bar R_i(t)) = \phi^{i-1}(\bar S_1^*,\bar E_1^*,\bar I_1^*,\bar R_1^*),\\
\lim_{t\to +\infty} (\bar S(t), \bar E(t), \bar I(t), \bar R(t))
= (\bar S^*, \bar E^*, \bar I^*, \bar R^*),
\end{gather*}
with
\begin{multline}
\label{eq84}
\bar S_1^*=\frac{b}{(\beta-\nu) \bar I^*+b},\quad
\bar E_1^*=\frac{\beta \bar I^*}{\sigma+b-\nu \bar I^*}\bar S_1^*,\\
 \bar I_1^*= \frac{\bar I^*}{\bar S^*} \bar S_1^*,\quad
\bar R_1^*= \frac{\gamma}{\omega+b-\nu \bar I^*}\frac{\bar I^*}{\bar S^*} \bar S_1^*.\hfill \square
\end{multline}
\end{enumerate}
\end{Thm}

\begin{proof}
The basic reproduction number of \eqref{eqn:SEIRS_normglobal} is easily computed by the method of next-generation matrix \cite{Driessche2002,Diekmann1990}.
Convergence to the equilibrium of the macroscopic system \eqref{eqn:SEIRS_normglobal} is proved in \cite{Greenhalgh1997} in the case $\cR_0<1$, and in \cite{Lu2018} in the case $\cR_0>1$.
In both cases, one proceeds recursively to prove the convergence to the equilibrium of the microscopic system, using  the convergence of system \eqref{eqn:SEIRS_normglobal}.
\end{proof}
\begin{Rem}
Theorem {\rm\ref{th7}} shows in particular that $\cR_0=1$ constitutes a threshold condition, which determines if the disease will eventually die out or remain in an endemic state with persistent (re)infections. 
\hfill $\square$
\end{Rem}
\begin{Rem}
\label{re5}
One may show from \eqref{eq85bis} and \eqref{eq84} that, for any $\bA, \bB\in\{\bS,\bE,\bI,\bR\}$, any $i,j\in\Nset$,
$\frac{\bA_i}{\bB_j} = \phi^{i-j} \frac{\bA}{\bB}$.
\hfill $\square$
\end{Rem}

The case $\nu=0$ yields simpler formulas, given now.

\begin{Cor}
\label{Cor1}
When $\nu=0$ and $\cR_0>1$, then
\[\cR_0=\frac{\beta}{\gamma+b}\frac{\sigma}{\sigma+b}, \quad \phi=\frac{\gamma\omega}{\beta(\omega+ b )-\gamma\omega}(\cR_0-1),\]
Moreover, we have an explicit formula for the value of $\bar I^*$ :
\begin{align*}
\bar I^* = \frac{\cR_0-1}{\cR_0}\frac{1}{\zeta},\ \zeta := \frac{(\gamma+b)(\sigma+b)(\omega+b)-\omega\gamma\sigma}{\sigma b(\omega+b)} > 1. \label{zeta} 
\end{align*}
The constant $ \zeta $ is 
called {\em number of critical stability} {\rm\cite{Cheng2012}}.\hfill $\square$
\end{Cor}

\begin{proof}
The 1st part of Corollary \ref{Cor1} is evident by Theorem \ref{th7}.
On the other hand, when $\nu=0$, then \eqref{eq85} reduces to an affine equation.
Finally, expanding the numerator in the expression of $\zeta$, shows easily that $\zeta>1$.
\end{proof}

We now derive from Theorem \ref{th7} a complete picture of the asymptotic behavior of the total size of the (non-normalized) population.

\begin{Thm}
\label{Pro1}
Let $N(t)$ be the total population size at time $t$ of the solution of system \eqref{eqn:SEIRS}
corresponding to a given initial condition fulfilling \eqref{eq73}.
The following properties hold.
\begin{enumerate}
    \item If $b<\mu$, then $N(t)$ converges to $0$ when $t \to +\infty$. \label{1}
    \item If $b=\mu$, \label{2}
    \begin{enumerate} 
        \item If $\nu=0$, then $N(t)=N(0)$ for all $t$. \label{2a}
        \item If $\nu>0$, 
        \begin{enumerate}
        \item $N(t)$ converges asymptotically to a positive finite limit when $\cR_0<1$. \label{2bi}
        \item $N(t)$ converges asymptotically to $0$ when $\cR_0>1$. \label{2bii}
        \end{enumerate}
    \end{enumerate}
    \item If $b>\mu$,
    \begin{enumerate}
        \item If $b-\mu> \nu \bar I^*$, then $N(t)$ tends to $+\infty$. \label{Pro13a}
        \item If $\cR_0>1$ and $b-\mu<\nu \bar I^*$, then $N(t)$ converges to $0$.
\hfill $\square$
        \label{Pro13b}
    \end{enumerate}
\end{enumerate}
\end{Thm}

As $\bar I^*$ does not depend upon $\mu$ (see Remark \ref{rem6}), there exist parameter sets fulfilling the case \ref{Pro13a}, resp.~\ref{Pro13b}.

\begin{proof}

\noindent $\bullet$
Assertions \ref{1}, \ref{2a} are evident from the identity $\dot N(t) = (b-\mu) N(t) - \nu I(t)$.
\ref{2bii} is proved thanks to Theorem \ref{th7}.

\noindent $\bullet$
For the case \ref{2bi}, showing that the Jacobian matrix of \eqref{eqn:SEIRS_normglobal} is Hurwitz demonstrates that $\bar I$ tends indeed exponentially to 0.
Writing $I = \bI N$, there thus exist $c_1,c_2>0$ such that
$-\nu c_1e^{-c_2t}N\leq \dot N  \leq 0$,
which yields by integration:
\[
\log N(t) \geq \log N(0)+\nu\frac{c_1}{c_2}(e^{-c_2t}-1) \geq \log N(0)-\nu \frac{c_1}{c_2},
\]
and finally :
$N(t)\geq N(0)e^{-\nu\frac{c_1}{c_2}}>0$,
for any $t\geq 0$.
$N(t)$ is thus decreasing and lower bounded by a positive quantity, so \ref{2bi} holds.

\noindent $\bullet$
When $b>\mu$ et $b-\mu>\nu \bar I^*$, then for any $\varepsilon>0$ and $t$ sufficiently large,
$\left(b-\mu-\nu (\bar I^*+\varepsilon)\right)N\leq \dot N$.
Taking $\varepsilon>0$ small enough in order to have $b-\mu-\nu (\bar I^*+\varepsilon)>0$, we prove the case \ref{Pro13a}.
Last, taking now $-\varepsilon$ instead of $\varepsilon$ and inverting the last inequality, yields similarly the case \ref{Pro13b}.
\end{proof}

\section{MEAN NUMBERS OF REINFECTIONS}
\label{se5}

Based on the convergence properties previously established, we now obtain
significant quantities, related to the asymptotic mean numbers of reinfections at endemic equilibrium.
\begin{Thm}
\label{th9}
Let $\cR_0>1$, then 
\vspace{-.1cm}
\begin{gather*}
\frac{\sum_{i\geq 1} (i-1)\bar S^*_i}{\sum_{i\geq 1} \bar S^*_i}
=\frac{\phi}{1-\phi},\\
\frac{\sum_{i\geq 1} i\bar E^*_i}{\sum_{i\geq 1} \bar E^*_i}=\frac{\sum_{i\geq 1} i\bar I^*_i}{\sum_{i\geq 1} \bar I^*_i}=\frac{\sum_{i\geq 1} i\bar R^*_i}{\sum_{i\geq 1} \bar R^*_i}=\frac{1}{1-\phi},
\end{gather*}
\vspace{-.5cm} 
\begin{multline*}
\frac{\sum_{i\geq 1} ((i-1)\bar S^*_i+ i(\bar E^*_i+\bar I^*_i+\bar R_i^*))}{\sum_{i\geq 1} (\bar S^*_i+\bar E^*_i+\bar I^*_i+\bar R^*_i)}\\
\hspace{2cm}
=\frac{1}{(1-\phi)^2}\frac{b-b\phi-\nu\bar I_1^*}{b-\nu\bar I^*}-\bar S^*.
\hfill \square
\end{multline*} 
\end{Thm}

\begin{proof}
According to Theorem \ref{th7}, the sequences $\{\bar S_i^*\}$, $\{\bar E_i^*\}$, $\{\bar I_i^*\}$ et $\{\bar R_i^*\}$ are geometric of common ratio $0<\phi<1$.
This allows to use geometric series formulas.
\end{proof}

The quantities considered in the previous statement are the mean numbers of infections undergone respectively by the susceptible individuals, by the non-susceptible individuals, and by the global population.

In absence of infection-induced mortality, we may find explicitly the exact value of $\bI^*$, yielding analytic expressions for the quantities considered in Theorem \ref{th9}.
\begin{Cor}
\label{pro2}
Suppose $\nu=0$ and $\cR_0>1$, then:
\label{eq219}
\vspace{-.1cm}
\begin{gather*}
\frac{\sum_{i\geq 1} (i-1)\bar S^*_i}{\sum_{i\geq 1} \bar S^*_i}
= \frac{\gamma\omega(\cR_0-1)}{\beta(\omega+b)-\gamma\omega\cR_0}, \label{eq219a}\\
\frac{\sum_{i\geq 1} i\bar E^*_i}{\sum_{i\geq 1} \bar E^*_i}=\frac{\sum_{i\geq 1} i\bar I^*_i}{\sum_{i\geq 1} \bar I^*_i}=\frac{\sum_{i\geq 1} i\bar R^*_i}{\sum_{i\geq 1} \bar R^*_i}=\frac{\beta(\omega+b)-\gamma\omega}{\beta(\omega+b)-\gamma\omega\cR_0},
\label{eq219b}
\end{gather*}
\begin{multline*}
\frac{\sum_{i\geq 1} ((i-1)\bar S^*_i+ i(\bar E^*_i+\bar I^*_i+\bar R_i^*))}{\sum_{i\geq 1} (\bar S^*_i+\bar E^*_i+\bar I^*_i+\bar R^*_i)}\\
\hspace{2cm}
= \frac{\beta(\omega+b)(\cR_0-1)}{(\beta(\omega+b)-\gamma\omega\cR_0)\cR_0}.
\hfill \square
\label{eq219c}
\end{multline*}
\end{Cor}
The proof of Corollary \ref{pro2} comes directly from applying Corollary \ref{Cor1} to the formulas in Theorem \ref{th9}.

\section{NUMERICAL SIMULATIONS}
\label{se6}

With the aim of illustrating the previous results, we present here some numerical simulations of system \eqref{eqn:SEIRS_norm}, for values of the coefficients borrowed from Bj\o rnstad et al.~\cite{Bjornstad2020}.
More precisely, we take $\cR_0=3$, $\gamma^{-1}=14$ days, $\sigma^{-1}=7$ days, $\omega^{-1}=1$ year, $b^{-1} = \mu^{-1}=76$ years, $\nu=0$ and $\beta=0.21$ days$^{-1}$.
The initial condition is chosen as $I^0=I^0_1=10^{-3}$, $S^0=S^0_1=1-I^0$, and all the other values null.
\begin{figure*}[ht]
\centering
\begin{minipage}[b]{0.45\linewidth}
\includegraphics[scale=0.4]{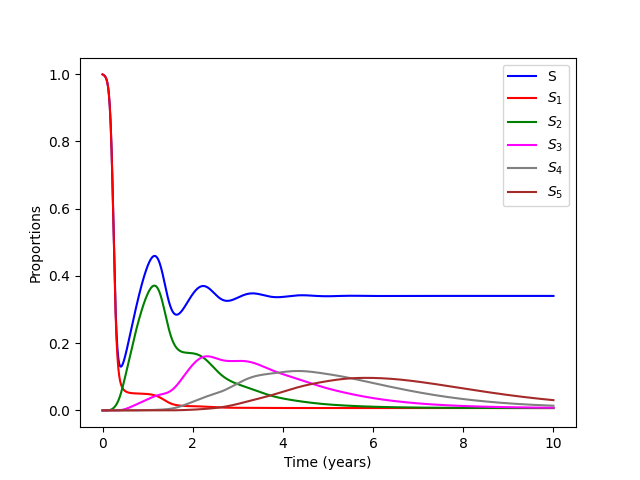}
\caption{Evolution of susceptible populations}
\label{fi3}
\end{minipage} 
\begin{minipage}[b]{0.45\linewidth}
\includegraphics[scale=0.4]{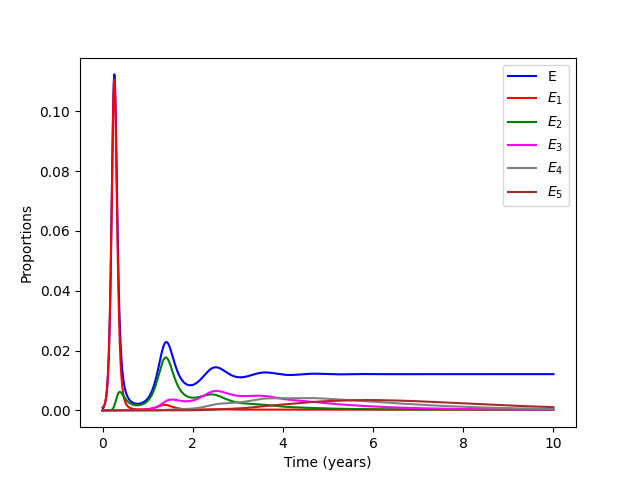}
\caption{Evolution of exposed populations}
\end{minipage}
\begin{minipage}[b]{0.45\linewidth}
\includegraphics[scale=0.4]{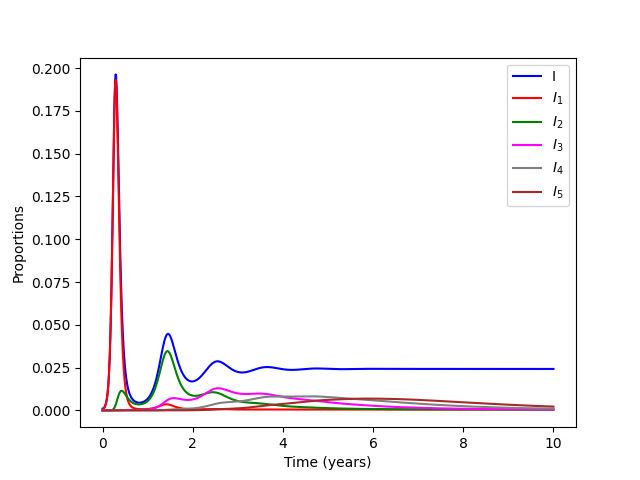}
\caption{Evolution of infected populations}
\label{fi5}
\end{minipage} 
\begin{minipage}[b]{0.45\linewidth}
\includegraphics[scale=0.4]{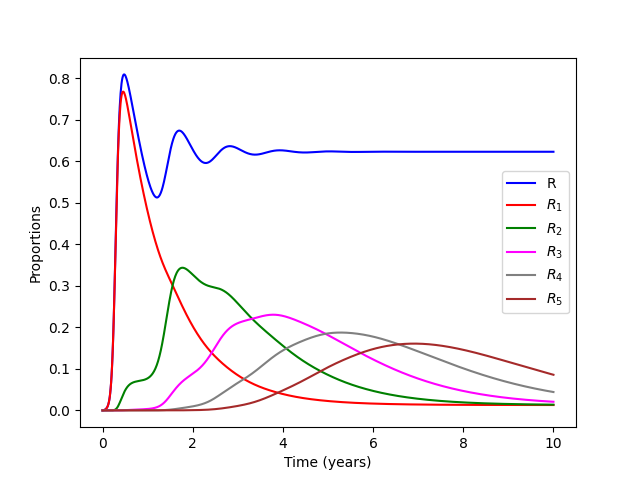}
\caption{Evolution of recovered populations}
\label{fi6}
\end{minipage}
\end{figure*}

The results are shown in Figures \ref{fi3}--\ref{fi6}.
They suggest that the convergence of the macroscopic components is slower than the convergence of each microscopic ones, in spite of the fact that the period during which these components evolve significantly begins and ends apparently later, and also later and later for higher numbers of reinfections.

One may check that, for this system with $b=\mu$ and $\nu=0$, one has at the endemic equilibrium
\[
\frac{\bS_1^*}{\bS^* }  = \frac{b\sigma\beta}{(\beta \bI^*+b)(\gamma+b)(\sigma+b)} \simeq 0.02.
\]
Also, the Jacobian matrix for the system \eqref{eqn:SEIRS_normglobal} describing the evolution of the macroscopic components is
\[
\begin{pmatrix}
-\beta\bar I^*-b &  0 & -\beta\bar S^* & \omega \\ \beta \bar I^* & -(\sigma+b) & \beta \bar S^*+\bar E^* & 0 \\ 0 & \sigma & -(\gamma+b) & 0 \\ 0 & 0 & \gamma & -(\omega+b)
\end{pmatrix},
\]
while all diagonal blocks of the block-triangular Jacobian matrix of the truncated system \eqref{eqn:SEIRS_norm}, corresponding to the evolution of a finite number of microscopic modes,
are worth
\[
\begin{pmatrix}
-\beta\bar I^*- b &  0 & 0 & 0 \\ \beta \bar I^* & -(\sigma+b) & 0 & 0 \\ 0 & \sigma & -(\gamma+b) & 0 \\ 0 & 0 & \gamma & -(\omega+b)
\end{pmatrix}.
\]
The latter matrix is diagonal, its spectrum is real, and is numerically approximated to
$\{ -1.01, -1.87$, $-26.08, -52.17\}$.
One computes numerically the spectrum of the former matrix, which appears to be complex, and approximately equal to $\{-1.28\times 10^{-2}, -68.62, -6.24 \pm 3.16 i \}$.
Both matrices are Hurwitz, and the largest real part of the eigenvalues is associated to the macroscopic evolution, as foreseen. 

\section{OBSERVABILITY AND IDENTIFIABILITY OF A SIMPLIFIED SIS MODEL}
\label{se7}

In order to illustrate the interest of the previous study for system identifiability, we consider here the following SIS system, formally obtained from
\eqref{eqn:SEIRS_norm} by putting $\nu=0$ and $\sigma, \omega \to+\infty$:
for any $i\geq 1$, 
\begin{subequations}
\label{eqn:SIS}
\begin{gather}
\label{eq1a}
  \hspace{-.17cm}  \dot{S}_{i}=\gamma I_{i-1}-\beta S_{i}I- \mu  S_{i},\ S_i(0)=S_i^0,\\
\label{eq1b}
    \dot{I}_{i}=\beta S_{i}I-( \mu +\gamma)I_{i}, \ I_i(0)=I_i^0.
    \end{gather}
\end{subequations}
One checks easily that the results obtained above are valid for this system.
Let us now show that the knowledge of the two positive limits $I^*$ and $I_1^*$ allows to compute all parameters of the system, provided the mortality rate $\mu$ is known.

\begin{Thm}
\label{th4}
Suppose the coefficient $\mu$ is known and that the limits of the numbers of infected $I^*>0$ and primo-infected $I_1^*>0$ are measured.
Then
\[\cR_0=\frac{1}{1-I^*},\]
and by posing
\begin{equation*}
\label{eq66a}
\theta:=\frac{I^*_1}{\sum_{i\geq 1} I^*_i} = \frac{I^*_1}{I^*},
\end{equation*}
the infection rate $\beta$ and the recovery rate $\gamma$ are given by:
\begin{equation*}
\label{eq72}
\beta
= \frac{\mu}{\cR_0-1}\left(
\frac{\cR_0^2}{\theta}-1
\right)\!,\ \gamma=\frac{\mu}{\cR_0-1}\left(
\frac{\cR_0}{\theta}-1
\right).\quad
\hfill \square
\end{equation*}
\end{Thm}

The demonstration, based on formulas given in Theorem~\ref{th9} presents no difficulty and is left to the reader.
Theorem~\ref{th4} provides a way to identify the coefficients $\beta$ and $\gamma$ of the system, when $\mu$ and the limit quantities $I_1^*$ and $I^*$ are measured.
Clearly, the knowledge of the two measurements brings more information than any of them does alone.

\begin{Rem}
The ratio $\theta$ is the proportion of primo-infected in the total infected population at endemic equilibrium.
It is as well the proportion of susceptible individuals never previously infected in the total susceptible population, see Remark {\rm\ref{re5}}.
\hfill $\square$
\end{Rem}

Theorem \ref{th4} 
 suggests to study the observability and identifiability properties 
 of the system.
To tackle this point, we study the following 4-dimensional system, obtained as subsystem of \eqref{eqn:SIS} (see the comment following Theorem \ref{th6}):
\begin{subequations}
\label{eq101}
\begin{gather}
\label{eq101a}
\dot S = \mu -\beta SI- \mu  S + \gamma I,\quad
\dot I =\beta SI- (\mu +\gamma)I,\\
\label{eq101b}
\dot S_1 = \mu -\beta S_1I- \mu  S_1,\quad
\dot I_1 =\beta S_1I- (\mu +\gamma)I_1,\\
\label{eq101c}
y := \alpha I,\qquad y_1 := \alpha I_1.
\end{gather}
\end{subequations}
We assume through \eqref{eq101c} that the measurements of a portion $y$ of the infected individuals $I$ is available (as done  e.g.~in \cite{Evans:2005wj}), as well as of a portion $y_1$ of the primo-infected $I_1$, {\em with the same proportion $\alpha$}.
The coefficient $\alpha$ lies in $(0,1]$.
As before, the mortality rate $\mu$ is supposed known, as well as the total population size, taken to 1 for simplicity.
The parameters $\alpha,\beta,\gamma$ are unknown.
\begin{Thm}
\label{th18}
When the measurement $y$ is available, then system \eqref{eq101} is neither observable, nor identifiable.
When both measurements $y$ and $y_1$ are available, then system \eqref{eq101} is both observable and identifiable.
\hfill $\square$
\end{Thm}

\begin{proof}
As $\nu=0$, we assume for simplicity $N\equiv 1$.

\noindent $\bullet$
{\bf Measuring only $y$.}
Consider first the use of the measurement $y$ on the macroscopic variables.
Due to the fact that $S+I\equiv 1$, the two formulas in \eqref{eq101a} provide indeed the same equation, namely:
$\dot I = ( \beta - (\mu+\gamma) - \beta I)I$.
Therefore, 
\begin{equation}
\label{eq103}
\dot y = \left(
\beta - (\mu+\gamma) - \frac{\beta}{\alpha} y
\right) y.
\end{equation}
By differentiation one gets that
\begin{equation*}
\label{eq110}
\frac{d}{dt} \left(
\frac{\dot y}{y}
\right) = -\frac{\beta}{\alpha} \dot y,
\end{equation*}
and one may express the two quantities $\frac{\beta}{\alpha}$ and $\beta-\gamma$, as
\begin{subequations}
\label{eq104}
\begin{equation}
\label{eq104a}
\frac{\beta}{\alpha}
= - \frac{1}{\dot y} \frac{d}{dt} \left(
\frac{\dot y}{y}
\right)
= - \frac{1}{\dot y} \frac{d^2}{dt^2} (\ln y),
\end{equation}
\vspace{-.5cm}
\begin{multline}
\label{eq104b}
\beta - \gamma
= \frac{\dot y}{y} + \mu + \frac{\beta}{\alpha} y
= \frac{\dot y}{y} + \mu - \frac{y}{\dot y} \frac{d^2}{dt^2} (\ln y)\\
= \mu + \frac{d}{dt} (\ln y) -  \frac{d}{dt} \left(
\ln \left|
\frac{d}{dt} (\ln y)
\right| \right).
\end{multline}
\end{subequations}
These two quantities are thus identifiable, but this is not sufficient to obtain each of the three coefficients $\alpha, \beta,\gamma$.
On the other hand, it is clear that nothing more may be learned when measuring only $y$, which fulfils equation \eqref{eq103}.
Therefore, the system \eqref{eq101a} 
is {\em not} identifiable over $\alpha, \beta,\gamma$.

Also, notice that it is not possible to determine $I = \frac{1}{\alpha}y$, otherwise $\alpha$ would be identifiable, and all other parameters too.
Thus the system is not observable.

\noindent $\bullet$
{\bf Measuring $y$ and $y_1$.}

We will now use \eqref{eq101b} and exploit the knowledge of the supplementary measured output $y_1$.
From the definition of $y_1$ and the second formula in \eqref{eq101b} one deduces, putting $w := \beta S_1$, that
\begin{equation}
\label{eq200}
\dot y_1 = (\beta S_1) y - (\mu+\gamma)y_1 = w y - (\mu+\gamma)y_1,
\end{equation}
and thus:
\begin{equation}
\label{eq201}
\ddot y_1 = \dot w y + w\dot y - (\mu+\gamma) \dot y_1.
\end{equation}

On the other hand, using now the first formula in \eqref{eq101b} and replacing $\frac{\beta}{\alpha}$ by its value obtained from \eqref{eq104a} yields
\begin{multline}
\label{eq202}
\dot w
= \beta \dot S_1
= \beta (\mu -\beta S_1I- \mu  S_1)
= \beta \mu - \frac{\beta}{\alpha} w y - \mu  w\\
= \beta \mu + \left(
\frac{d}{dt} \left(
\ln \left|
\frac{d}{dt} (\ln y)
\right| \right)
 - \mu
 \right)  w.
\end{multline}

Inserting in \eqref{eq201} the value of $\dot w$ extracted by \eqref{eq202}, we get
\[
\ddot y_1 = \left(
\beta \mu + \left(
\frac{d}{dt} \left(
\ln \left|
\frac{d}{dt} (\ln y)
\right| \right)
 - \mu
 \right)  w \right) y + w\dot y - (\mu+\gamma) \dot y_1,
\]
and gathering the terms in $w$ leads to the following equivalent form:
\begin{equation}
\label{eq203}
\ddot y_1
=
\beta \mu y + \left(
\dot y  + y\frac{d}{dt} \left(
\ln \left|
\frac{d}{dt} (\ln y)
\right| \right) - \mu y
\right) w
- (\mu+\gamma) \dot y_1.
\end{equation}

For clarity, let us write \eqref{eq200} and \eqref{eq203} under matrix form:
\begin{multline*}
\begin{pmatrix}
-y_1 & y\\
- \dot y_1 & \dot y  + y\frac{d}{dt} \left(
\ln \left|
\frac{d}{dt} (\ln y)
\right| \right) - \mu y
\end{pmatrix}
\begin{pmatrix}
\mu+\gamma \\ w
\end{pmatrix}\\
= \begin{pmatrix}
\dot y_1 \\ \ddot y_1 - \beta \mu y
\end{pmatrix}.
\end{multline*}
One deduces by partial inversion of the matrix that
\begin{equation}
\label{eq300}
\mu+\gamma
= \frac{\begin{vmatrix}
\dot y_1 & y\\
\ddot y_1 - \beta \mu y & \dot y  + y\frac{d}{dt} \left(
\ln \left|
\frac{d}{dt} (\ln y)
\right| \right) - \mu y
\end{vmatrix}}{\begin{vmatrix}
-y_1 & y\\
- \dot y_1 & \dot y  + y\frac{d}{dt} \left(
\ln \left|
\frac{d}{dt} (\ln y)
\right| \right) - \mu y
\end{vmatrix}}.
\end{equation}

On the other hand, one has, by use of \eqref{eq104b}, an alternative expression of $\mu+\gamma$, namely
\begin{equation}
\label{eq301}
\mu+\gamma
= \beta - \frac{d}{dt} (\ln y) + \frac{d}{dt} \left(
\ln \left|
\frac{d}{dt} (\ln y)
\right| \right).
\end{equation}
Achieving elimination of $\mu+\gamma$ between \eqref{eq300} and \eqref{eq301} yields
\begin{multline*}
\beta - \frac{d}{dt} (\ln y) + \frac{d}{dt} \left(
\ln \left|
\frac{d}{dt} (\ln y)
\right| \right)\\
=
\frac{\begin{vmatrix}
\dot y_1 & y\\
\ddot y_1 - \beta \mu y & \dot y  + y\frac{d}{dt} \left(
\ln \left|
\frac{d}{dt} (\ln y)
\right| \right) - \mu y
\end{vmatrix}}{\begin{vmatrix}
-y_1 & y\\
- \dot y_1 & \dot y  + y\frac{d}{dt} \left(
\ln \left|
\frac{d}{dt} (\ln y)
\right| \right) - \mu y
\end{vmatrix}}.
\end{multline*}
Now, one checks that the previous identity is indeed {\em affine in $\beta$}.
Factorizing all terms in $\beta$, it may be written
\begin{multline}
\label{eq576}
\left(
- \left(
\dot y  + y\frac{d}{dt} \left(
\ln \left|
\frac{d}{dt} (\ln y)
\right| \right) - \mu y
\right) y_1+ y \dot y_1  - \mu y^2
\right) \beta\\
= \Phi(y,\dot y,\ddot y,y_1,\dot y_1,\ddot y_1),
\end{multline}
where the function $\Phi$ depends upon the two outputs $y$ and $y_1$ and their derivatives up to the second one.
In particular, $\Phi$ does not contain any occurrence of the unknown coefficients $\alpha,\beta,\gamma$.
One deduces that $\beta$ may be identified, and thus also $\alpha, \gamma$ with the help of \eqref{eq104}.
Therefore in these conditions, system \eqref{eq101} is identifiable\footnote{Typically, the numbers of infected $I(t)$  and primo-infected $I_1(t)$ are almost identical for small values of time $t$ (see e.g.~Figure \ref{fi5}).
Putting $y_1 \sim y$  and $\dot y_1 \sim \dot y$ in the factor in the left-hand size of \eqref{eq576} gives the value $-y^2\frac{d}{dt} \left(
\ln \left|
\frac{d}{dt} (\ln y)
\right| \right) = -y^2 \frac{d}{dt} \left( \frac{\dot y}{y} \right) = -y^2 \frac{d}{dt} \ln \left(
\beta - (\mu+\gamma) - \frac{\beta}{\alpha} y
\right)$ along the trajectories of \eqref{eq103}, which generically is nonzero.}.

Once these parameters have been identified, one has $I = \frac{1}{\alpha} y$ and $S=1-I$, while on the other hand, $I_1 = \frac{1}{\alpha} y_1$ and $S_1 = \frac{1}{\beta} w$, where (see \eqref{eq200}) $w$ is given by
\begin{equation*}
w 
= \frac{y_1}{y} \left(
\frac{d}{dt} (\ln y_1) + \mu+\gamma
\right).
\end{equation*}
The system \eqref{eq101} is thus observable.
\end{proof}

 \addtolength{\textheight}{-8cm}   

\section{CONCLUSIONS}
\label{se8}

We proposed in this article a SEIRS model with an infinite set of differential equations, allowing to enumerate the number of reinfections.
The well-posedness of this system has been established in an appropriate functional setting, and the asymptotic convergence to either the disease-free equilibrium (when the basic reproduction number $\cR_0$ is smaller than 1) or the endemic equilibrium (when $\cR_0>1$) has been shown.
We also provided in the latter case several formulas related to mean numbers of reinfections.
Last, we have shown that the joint measurement of the number of infected and of primo-infected is sufficient to render observable and identifiable a system that is not when only the infected are measured.
This result demonstrates the interest of the reinfection data for analyzing the communicable diseases. 
Based on this first step, further research will now consider the key issues of observation and identification.





\bibliographystyle{plain}
\bibliography{BIBLIO.bib}

\end{document}